\RequirePackage{fix-cm}
\documentclass[smallextended]{svjour3}       
\smartqed  
\usepackage{amssymb,amsmath,amsfonts,mathrsfs}
\usepackage[utf8]{inputenc}
\usepackage[english]{babel}
\usepackage{graphicx}
\DeclareMathOperator{\rank}{rank}

\DeclareMathOperator{\inth}{int.hull}

\DeclareMathOperator{\size}{size}

\DeclareMathOperator{\poly}{poly}
\DeclareMathOperator{\mult}{Mult}

\DeclareMathOperator{\vol}{Vol}

\begin{document}

\title{FPT-algorithms for The Shortest Lattice Vector and Integer Linear Programming Problems}


\author{D. V. Gribanov}


\institute{D. V. Gribanov \at Lobachevsky State University of Nizhny Novgorod, 23 Gagarina Avenue, Nizhny Novgorod, Russian Federation, 603950,\\
National Research University Higher School of Economics, 136 Rodionova, Nizhny Novgorod, Russian Federation, 603093,\\
\email{dimitry.gribanov@gmail.com}
}

\date{Received: date / Accepted: date}

\maketitle

\begin{abstract}
In this paper, we present FPT-algorithms for special cases of the shortest vector problem (SVP) and the integer linear programming problem (ILP), when matrices included to the problems' formulations are near square. The main parameter is the maximal absolute value of rank minors of matrices included to the problem formulation. Additionally, we present FPT-algorithms with respect to the same main parameter for the problems, when the matrices have no singular rank sub-matrices.

\end{abstract}

\section{Introduction}
Let $A \in \mathbb{Z}^{d \times n}$ be the integral matrix. Its $ij$-th element is
denoted by $A_{i\,j}$, $A_{i\,*}$ is $i$-th row of $A$, and $A_{*\,j}$ is $j$-th column of $A$. The set of integer values started from a value $i$ and finished on $j$ is denoted by the symbol $i:j = \{i,i+1,\dots,j\}$. Additionally, for subsets $I \subseteq \{1,\dots,d\}$ and $J \subseteq \{1,\dots,n\}$, $A_{I\,J}$ denotes the sub-matrix of $A$ that was generated by all rows with numbers in $I$ and all columns with numbers in $J$. Sometimes, we will change the symbols $I$ and $J$ to the symbol $*$ meaning that we take the set of all rows or columns, respectively. Let $\rank(A)$  be the rank of an integral matrix $A$. The \emph{lattice spanned by columns} of $A$ is denoted $\Lambda(A) = \{A t : t \in \mathbb{Z}^{n} \}$. Let $||A||_{\max}$ denote the maximal absolute value of the elements of $A$. We refer to \cite{CAS71,GRUB87,SIEG89} for mathematical introductions to lattices.

An algorithm parameterized by a parameter $k$ is called \emph{fixed-parameter tractable} (FPT-\emph{algorithm}) if its complexity can be expressed by a function from the class $f(k)\, n^{O(1)}$, where $n$ is the input size and $f(k)$ is a function that depends on $k$ only.
A computational problem parameterized by a parameter $k$ is called \emph{fixed-parameter tractable} (FPT-\emph{problem}) if it can be solved by a FPT-algorithm. For more information about parameterized complexity theory, see \cite{PARAM15,PARAM99}.

{\bf Shortest Lattice Vector Problem}

The Shortest Lattice Vector Problem (SVP) consists in finding $x \in \mathbb{Z}^n \setminus \{0\}$ minimizing $||H x||$, where $H \in \mathbb{Q}^{d \times n}$ is given as an input. The SVP is known to be NP-hard with respect to randomized reductions, see \cite{AJTAI96}. The first polynomial-time approximation algorithm for SVP was proposed by A.~Lenstra, H.~Lenstra~Jr. and L.~Lov\'asz in the paper \cite{LLL82}.  Shortly afterwards, U.~Fincke and M.~Pohst \cite{FP83,FP85}, and R.~Kannan \cite{KANN83,KANN87} described the first exact SVP solvers. The R.~Kannan's solver has the complexity $2^{O(n\,\log n)} \poly(\size H)$. The first SVP solvers that achieve the complexity $2^{O(n)} \poly(\size H)$ were proposed by M.~Ajtai, R.~Kumar and D.~Sivakumar \cite{AJKSK01,AJKSK02}, and D.~Micciancio and P.~Voulgaris \cite{MICCVOUL10}. The previously discussed SVP solvers are useful for the $l_2$ Euclidean norm. Recent results about SVP-solvers for more general norms are presented in the papers \cite{BLNAEW09,DAD11,EIS11}. The paper of G.~Hanrot, X.~Pujol, D.~Stehl\'e \cite{SVPSUR11} is a good survey about SVP-solvers. 

Recently, a novel polynomial-time approximation SVP-solver was proposed by J.~Cheon and L.~Changmin in the paper \cite{CHLEE15}. The algorithm is parameterized by the lattice determinant, and its complexity and approximation factor are record for bounded determinant lattices.

In our work, we consider only integral lattices, whose generating matrices are near square. The goal of Section 2 is development of an exact FPT-algorithm for the SVP parameterized by the lattice determinant. Additionally, in Section 3 we develop a FPT-algorithm for lattices, whose generating matrices have no singular sub-matrices. The proposed algorithms work for the $l_p$ norm for any finite $p \geq 1$.

{\bf Integer Linear Programming Problem}

The Integer Linear Programming Problem (ILPP) can be formulated as $\min\{ c^\top x : H x \leq b,\, x \in \mathbb{Z}^n\}$ for integral vectors $c,b$ and an integral matrix $H$.

There are several polynomial-time algorithms for solving the linear programs. We mention L.~G.~Khachiyan's algorithm \cite{KHA80}, N.~Karmarkar's algorithm \cite{KAR84}, and Y.~E.~Nesterov's algorithm \cite{NN94,PAR91}. Unfortunately, it is well known that the ILPP is NP-hard problem. Therefore, it would be interesting to reveal polynomially solvable cases of the ILPP. Recall that an integer matrix is called \emph{totally unimodular} if any of its minor is equal to $+1$ or $-1$ or $0$. It is well known that all optimal solutions of any linear program with a totally unimodular constraint matrix are integer. Hence, for any primal linear program and the corresponding primal integer linear program with a totally unimodular constraint matrix, the sets of their optimal solutions coincide. Therefore, any polynomial-time linear optimization algorithm (like algorithms in \cite{KHA80,KAR84,NN94,PAR91}) is also an efficient algorithm for the ILPP.

The next natural step is to consider the \emph{bimodular} case, i.e. the ILPP having constraint matrices with the absolute values of all rank minors in the set $\{0, 1, 2\}$. The first paper that discovers fundamental properties of the bimodular ILPP is the paper of S.~Veselov and A.~Chirkov \cite{VESCH09}. Very recently, using results of \cite{VESCH09}, a strongly polynomial-time solvability of the bimodular ILPP was proved by S.~Artmann, R.~Weismantel, R.~Zenklusen in the paper \cite{AW17}.  

More generally, it would be interesting to investigate the complexity of the problems with constraint matrices having bounded minors. The maximum absolute value of rank minors of an integer matrix can be interpreted as a proximity measure to the class of unimodular matrices. Let the symbol ILPP$_{\Delta}$ denote the ILPP with constraint matrix each rank minor of which has the absolute value at most $\Delta$. A conjecture arises that for each fixed natural number $\Delta$ the ILPP$_{\Delta}$ can be solved in polynomial-time \cite{SHEV96}. There are variants of this conjecture, where the augmented matrices $\dbinom{c^\top}{A}$ and $(A \; b)$ are considered \cite{AZ11,SHEV96}.  Unfortunately, not much is known about the complexity of the ILPP$_{\Delta}$. For example, the complexity statuses of the ILPP$_{3}$ are unknown. A next step towards a clarification of the complexity was done by S.~Artmann, F.~Eisenbrand, C.~Glanzer, O.~Timm, S.~Vempala, and R.~Weismantel in the paper \cite{AE16}. Namely, it has been shown that if the constraint matrix, additionally, has no singular rank sub-matrices, then the ILPP$_{\Delta}$ can be solved in polynomial-time. Some results about polynomial-time solvability of boolean ILPP$_{\Delta}$ were obtained in the papers \cite{AZ11,BOCK14,GRIBM17}. Additionally, the class of ILPP$_{\Delta}$ has a set of interesting properties. In the papers \cite{GRIB13,GRIBV16}, it has been shown that any lattice-free polyhedron of the ILPP$_{\Delta}$ has relatively small width, i.e., the width is bounded by a function that is linear by the dimension and exponential by $\Delta$. Interestingly, due to \cite{GRIBV16}, the width of an empty lattice simplex can be estimated by $\Delta$ for this case. In the paper \cite{GRIBC16}, it has been shown that the width of any simplex induced by a system with bounded minors can be computed by a polynomial-time algorithm. Additional result of \cite{GRIBC16} states that any simple cone can be represented as a union of $n^{2 \log \Delta}$ unimodular cones, where $\Delta$ is the parameter that bounds minors of the cone constraint matrix. As it was mentioned in \cite{AW17}, due to E.~Tardos results \cite{TAR86}, linear programs with constraint matrices whose minors are bounded by a constant $\Delta$ can be solved in strongly polynomial time. N.~Bonifas et al. \cite{BONY14} showed that polyhedra defined by a constraint matrix that is totally $\Delta$-modular have small diameter, i.e., the diameter is bounded by a polynomial in $\Delta$ and the number of variables. Very recently, F.~Eisenbrand and S.~Vempala \cite{EIS16} showed a randomized simplex-type linear programming algorithm, whose running time is strongly polynomial even if all minors of the constraint matrix are bounded by any constant.

The second goal of our paper (Section 4) is to improve results of the paper \cite{AE16}. Namely, we are going to present a FPT-algorithm for the ILPP$_{\Delta}$ with the additional property that the problem's constraint matrix has no singular rank sub-matrices. Additionally, we improve some inequalities established in \cite{AE16}.

The authors consider this paper as a part for achieving the general aim to find out critical values of parameters, when a given problem changes complexity. For example, the integer linear programming problem is polynomial-time solvable on polyhedrons with all-integer vertices, due to \cite{KHA80}. On the other hand, it is NP-complete in the class of polyhedrons with denominators of extreme points equal $1$ or $2$, due \cite{PAD89}. The famous $k$-satisfiability problem is polynomial for $k \leq 2$, but is NP-complete for all $k > 2$. A theory, when an NP-complete graph problem becomes easier, is investigated for the family of hereditary classes in the papers \cite{A,ABKL,AKL,KLMT,M1,M2,M3,M4,MP,M5,MS}.

\section{FPT-algorithm for the SVP}

Let $H \in \mathbb{Z}^{d \times n}$. The SVP related to the $L_p$ norm can be formulated as follows:

\begin{equation}\label{ISVP}
\min\limits_{x \in \Lambda(H) \setminus \{0\} } ||x||_p,
\end{equation} or equivalently 
\[
\min\limits_{x \in \mathbb{Z}^n \setminus \{0\} } ||H x||_p.
\]

Without loss of generality, we can assume that the following properties hold:

1) the matrix $H$ is already reduced to the Hermite normal form (HNF) \cite{SCHR98,STORH96,ZHEN05},

2) the matrix $H$ is a full rank matrix and $d \geq n$,

3) using additional permutations of rows and columns, the HNF of the matrix $H$ can be reduced to the following form: 

\begin{equation} \label{HNF}
H = \begin{pmatrix}
1            & 0                   & \dots         & 0           & 0          & 0            & \dots & 0\\
0            & 1                   & \dots         & 0           & 0          & 0            & \dots & 0\\
\hdotsfor{8} \\
0            &        0            & \dots         & 1           & 0          & 0            & \dots & 0\\
a_{1\,1}  &   a_{1\,2}        & \dots        & a_{1\,k}  & b_{1\,1} & 0           & \dots & 0\\
a_{2\,1}  &   a_{1\,2}        & \dots        & a_{2\,k}  & b_{2\,1} & b_{2\,2} & \dots & 0\\
\hdotsfor{8} \\
a_{s\,1}  &   a_{1\,2}        & \dots        & a_{s\,k}   & b_{s\,1} & b_{s\,2} & \dots & b_{s\,s}\\
\bar a_{1\,1} & \bar a_{1\,2}           & \dots        & \bar a_{1\,k}   & \bar b_{1\,1} & \bar b_{1\,2} & \dots & \bar b_{1\,s}\\
\hdotsfor{8} \\
\bar a_{m\,1} & \bar a_{m\,2}           & \dots        & \bar a_{m\,k}   & \bar b_{m\,1} & \bar b_{m\,2} & \dots & \bar b_{m\,s}\\
\end{pmatrix},
\end{equation}

where $k + s = n$ and $k + s + m = d$. 

Let $\Delta$ be the maximal absolute value of $n \times n$ minors of $H$ and let $\delta = |\det(A_{1:n\,*})|$, let also $A \in \mathbb{Z}^{s \times k}$, $B \in \mathbb{Z}^{s \times s}$, $\bar A \in \mathbb{Z}^{m \times k}$, and $\bar B \in \mathbb{Z}^{m \times s}$ be the matrices defined by the elements $\{a_{i\,j}\}$, $\{b_{i\,j}\}$, $\{\bar a_{i\,j}\}$, and $\{\bar b_{i\,j}\}$, respectively. Hence, $B$ is lower triangular.

The following properties are standard for the HNF of any matrix:

1) $0 \leq a_{i\,j} \leq b_{i\,i}$ for any $i \in 1:s$ and $j \in 1:k$,

2) $0 \leq b_{i\,j} \leq b_{i\,i}$ for any $i \in 1:s$ and $j \in 1:i$,

3) $\Delta \geq \delta = \prod_{i=1}^s b_{i\,i}$, and hence $ s \leq \log_2 \Delta$.

In the paper \cite{AE16}, it was showed that $||(\bar A \; \bar B)||_{\max} \leq B_q$, where $q = \lceil \log_2 \Delta \rceil$ and the sequence $\{B_i\}$ is defined for $i \in 0:q$ as follows:
\[
B_0 = \Delta,\quad B_i = \Delta + \sum_{j=0}^{i-1} B_j \Delta^{\log_2 \Delta} (\log_2 \Delta)^{(\log_2 \Delta /2)}.
\] 

It is easy to see that $B_q = \Delta (\Delta^{\log_2 \Delta} (\log_2 \Delta)^{(\log_2 \Delta /2)}+1)^{\lceil \log_2 \Delta \rceil}$.

We will show that the estimate on $||(\bar A \; \bar B)||_{\max}$ can be significantly improved by making a bit more accurate analysis as in \cite{AE16}.

\begin{lemma}\label{HNFElem}
Let $j \in 1:m$, then the following inequalities are true:
\[
\bar b_{j\,i} \leq \frac{\Delta}{2} (3^{s-i} + 1)
\] for $i \in 1:s$, and
\[
\bar a_{j\,i} \leq \frac{\Delta}{2} (3^s + 1)
\] for $i \in 1:k$.

Hence, $||(\bar A \; \bar B)||_{\max} \leq \frac{\Delta}{2} (\Delta^{\log_2 3} + 1) < \Delta^{1 + \log_2 3}$.
\end{lemma} 
\begin{proof}
The main idea and the skeleton of the proof is the same as in the paper \cite{AE16}.

Assume that $H$ has the form as in \eqref{HNF}. Let $c$ be any row of $\bar A$, let also $w$ be the row of $\bar B$ with the same row index as $c$.

Let $H_i$ denote the square sub-matrix of $H$ that consists of the first $n$ rows of $H$, except the $i$-th row, which is replaced by the row $(c\;w)$. Let also $b_i$ denote $b_{i\,i}$. Since $|det(H_n)| = b_{1}\dots b_{s-1} |w_s|$, it follows that $|w_s| \leq \Delta$.

Similar to reasonings of the paper \cite{AE16}, let us consider two cases:

{\bf Case 1: $i > k$.}

We can express $det(H_i)$ as follows:
\[
|b_1|\dots|b_{r-1}| | \det \underbrace{\begin{pmatrix}
w_{r} & \hdotsfor{4} & w_s \\
* & b_{r+1} & & & & \\
* & * & \ddots & & & \\
\hdotsfor{4} & b_{s-1} & \\
\hdotsfor{5} & b_s \\
\end{pmatrix}}_{:= \bar H} |,
\]
where $r=i - k$.

Let ${\bar H}^j$ be the sub-matrix of $\bar H$ obtained by deletion of the first row and the column indexed by $j$. Then,
\[
\Delta \geq |\det \bar H| = \left| w_r {\bar H}^1 + \sum_{j=2}^{s-r+1} (-1)^{j+1} w_{r+j-1} {\bar H}^j  \right| \\
\geq |w_r {\bar H}^1| - |\sum_{j=2}^{s-r+1} (-1)^{j+1} w_{r+j-1} {\bar H}^j|,
\]

and thus 
\[
|w_r| \leq \frac{1}{|\det {\bar H}^1|} \left(\Delta + \sum_{j=2}^{s-r+1} |w_{r+j-1}| |{\bar H}^j| \right).
\]

Let $ \bar \delta = |\det {\bar H}^1| = b_{r+1}\dots b_s$. 
We note that for any $2 \leq j \leq s-r+1$ the matrix ${\bar H}^j$ is a lower-triangular matrix with an additional over-diagonal. The over-diagonal is a vector that consists of at most $j-2$ first nonzeros and the other elements are zeros.

The following example expresses the structure of the ${\bar H}^5$ matrix:
\[
\begin{pmatrix}
\boldsymbol{*} &*  &    &  & & & &\\
*&  \boldsymbol{*} & *  &  & & & &\\
* & * & \boldsymbol{*} & *  & & & &\\
* & * & * & \boldsymbol{*} & & & &\\
* & * & * & * & \boldsymbol{*} & & &\\
\hdotsfor{5} \\
* & \hdotsfor{6} &\boldsymbol{*} & \\
* & \hdotsfor{7} & \boldsymbol{*}\\

\end{pmatrix},
\]
where we can see three additional nonzero over-diagonal elements (the diagonal is bold).

It is easy to see that $|\det {\bar H}^j| \leq 2^{j-2} \bar \delta$ for $2 \leq j \leq s-r+1$. Hence, the recurrence for $w_r$ takes the following form:
\[
|w_r| \leq \Delta + \sum_{j=2}^{s-r+1} 2^{j-2} |w_{r+j-1}| = \Delta + \sum_{j=0}^{s-r-1} 2^{j} |w_{r+j+1}|.
\]

{\bf Case 2: $i \leq k$.} 

Similar to the previous case, we can express $|\det H_i|$ as
\[
|\det \underbrace{\begin{pmatrix}
c_i & w_1 & \dots & \dots & w_s \\
*   & b_1       \\
*   & *    & \ddots \\
\hdotsfor{3} & b_{s-1} \\
\hdotsfor{4} & b_s \\
\end{pmatrix}|}_{:= \bar H}.
\] 

Let again, ${\bar H}^j$ be the sub-matrix of the matrix $\bar H$ obtained by deletion of the first row and the column indexed by $j$, then
\[
\Delta \geq |\det \bar H| = \left| c_i \delta + \sum_{j=2}^{s+1} (-1)^{j+1} w_{j-1} {\bar H}^j \right|
\geq |c_i \delta| - |\sum_{j=2}^{s+1} (-1)^{j+1} w_{j-1} {\bar H}^j|,
\]
and thus 
\[
|c_i| \leq \frac{1}{\delta} \left(\Delta + \sum_{j=2}^{s+1} |w_{j-1}| |{\bar H}^j| \right).
\]

As in the {\bf case 1}, we have the inequality $|\det {\bar H}^j| \leq 2^{j-2} \delta$ for $2 \leq j \leq s+1$. Hence, we have the following inequality:
\[
|c_i| \leq \Delta + \sum_{j=2}^{s+1} 2^{j-2} |w_{j-1}| = \Delta + \sum_{j=0}^{s-1} 2^j |w_{j+1}|.
\] 

Let $\{B_i\}_{i=0}^{s}$ be the sequence defined as follows:
\[
B_0 = \Delta,\quad B_i = \Delta + \sum_{j=0}^{i-1} 2^{i-j-1} B_j.
\]

Using the final inequality from Case 1, we have $|w_i| \leq B_{s-i}$ for any $i \in 1:s$. And using the final inequality from Case 2, we have $|c_i| \leq B_s$ for any $i \in 1:k$.

For the sequence $\{B_i\}$ the following equalities are true:
\[
B_i = \Delta + B_{i-1} + \sum_{j=0}^{i-2} 2^{i-j-1} B_j = \Delta + B_{i-1} + 2(B_{i-1} - \Delta) = 3 B_{i-1} - \Delta.
\]

Finally,
\[
B_i = 3^i B_0 - \Delta \sum_{j=0}^{i-1} 3^j = \Delta( 3^i - \frac{3^i - 1}{2}) = \frac{\Delta}{2}(3^i + 1),
\]
and the lemma follows. 
\end{proof}

\begin{theorem}\label{SimpleSVP}
If $n > \Delta^{1+m(1+\log_2 3)} + \log_{2} \Delta$, then there exists a polynomial-time algorithm to solve the problem \eqref{ISVP} with the bit-complexity $O(n \log n \log \Delta (m + \log \Delta))$.
\end{theorem}

\begin{proof}
If $n > \Delta^{1+m(1+\log_2 3)} + \log_2 \Delta$, then $k > \Delta^{1+m(1+\log_2 3)}$. 

Consider the matrix $\bar H = \dbinom{A}{\bar A}$. By Lemma \ref{HNFElem}, there are strictly less than $\Delta^{1+m(1+\log_2 3)}$ possibilities to generate a column from $\bar A$, so if $k > \Delta^{1+m(1+\log_2 3)}$, then $\bar H$ has two equivalent columns. Hence, the lattice $\Lambda(H)$ contains the vector $v$, such that $||v||_p = \sqrt[p]{2}$ ($||v||_\infty = 1$). We can find equivalent rows using any sorting algorithm with the compares-complexity equal to $O(n \log n)$, where the bit-complexity of the two vectors compare operation is $O(\log \Delta (m + \log \Delta))$. The lattice $\Lambda(H)$ contains a vector of the norm $1$ (for $p \not= \infty$) if and only if the matrix $\bar H$ contains the zero column. Definitely, let $\bar H$ have no zero columns and let $u$ be the vector of the norm $1$ induced by the lattice $\Lambda(H)$. Then, $u \in H_{*\,i} + H_{*\,(k+1):n} t$ for the integral, nonzero vector $t$ and $i \in 1:k$. Let $j \in 1:(n-k)$ be the first index, such that $t_j \not= 0$. Since $H_{j\,j} > H_{j\,i}$, we have $u_j \not= 0$ and $||u||_p \geq  \sqrt[p]{2}$, this is a contradiction.  
\end{proof}

In the case, when $m = 0$ and $H$ is the square nonsingular matrix, we have the following trivial corollary:
\begin{corollary}
If $n \geq \Delta + \log_{2} \Delta$, then there exists a polynomial-time algorithm to solve the problem \eqref{ISVP} with the bit-complexity $O(n \log n \log^2 \Delta)$.
\end{corollary}

Let $x^*$ be an optimal vector of the problem \eqref{ISVP}. The most classical G.~Minkowski's theorem in geometry of numbers states that:
\[
||x^*||_p \leq 2 \left(\frac{\det \Lambda(H)}{\vol(B_p)}\right)^{1/n},
\]
where $B_p$ is the unit sphere for the $l_p$ norm.

Using the inequalities $\det \Lambda(H) = \sqrt{\det H^\top H} \leq \Delta \sqrt{\dbinom{d}{n}} \leq \Delta \left(\cfrac{e d}{n}\right)^{n/2}$, we can conclude  that
\[
||x^*||_p \leq 2 \sqrt{\frac{e d}{n}}  \sqrt[n]{\frac{\Delta}{\vol(B_p)}}.
\]  

On the other hand, by Lemma \ref{HNFElem}, the last column of $H$ has the norm equals $\Delta \sqrt[p]{m+1}$. 

Let 
\begin{equation}\label{MConst} 
M = \min \Bigl\{\Delta \sqrt[p]{m+1},\, 2 \sqrt{\frac{e d}{n}}  \sqrt[n]{\frac{\Delta}{\vol(B_p)}}\Bigr\}
\end{equation} be the minimum value between these two estimates on a shortest vector norm. 

\begin{lemma}\label{SVBounds}
Let $x^* = \binom{\alpha}{\beta}$ be an optimal solution of \eqref{ISVP}, then:

1) $||\alpha||_1 \leq M^p$,

2) $|\beta_i| \leq 2^{i-1}(M^p + M/2)$, for any $i \in 1:s$,

3) $||\beta||_1 \leq 2^s (M^p + M/2) \leq \Delta (M^p + M/2) < 2 \Delta M^p$ 

and $||x^*||_1 \leq (1 + \Delta) M^p + \frac{\Delta M}{2} < 2 (1 + \Delta)M^p$.
\end{lemma}

\begin{proof}
The statement 1) is trivial. 

For $\beta_1$ we have:
\[
|b_{1\,1} \beta_{1} + \sum_{i=1}^{k} a_{1\,i} \alpha_{i}| \leq M,
\]
\[
\sum_{i=1}^{k} a_{1\,i} \alpha_{i} - M \leq b_{1\,1} \beta_{1} \leq \sum_{i=1}^{k} a_{1\,i} \alpha_{i} + M,
\]
\[
- M^p - M/2 \leq \frac{1}{b_{1\,1}} (\sum_{i=1}^{k} a_{1\,i} \alpha_{i} - M) \leq \beta_{1} \leq \frac{1}{b_{1\,1}} (\sum_{i=1}^{k} a_{1\,i} \alpha_{i} + M) \leq M^p + M/2.
\]

For $\beta_j$ we have:
\[
|\sum_{i=1}^{j} b_{j\,i} \beta_{i} + \sum_{i=1}^{k} a_{j\,i} \alpha_{i}| \leq M,
\]
\[
\beta_{j} \leq \frac{1}{b_{j\,j}} (\sum_{i=1}^{j-1} b_{j\,i} \beta_{i} + \sum_{i=1}^{k} a_{j\,i} \alpha_{i} + M) \leq \sum_{i=1}^{j-1} \beta_{i} +  M^p + M/2,
\]
\[
|\beta_{j}| \leq  2^{j-1} (M^p + M/2).
\]

The statement 3) follows from the proposition 2).
\end{proof}

Let $Prob(l,v,u,C)$ denote the following problem:
\begin{align}
&\sum_{i=1}^l |\alpha_i|^p + \sum_{j=1}^s \left| \sum_{i=1}^l a_{j\,i} \alpha_i + v_j \right|^p + \sum_{j=1}^m \left|\sum_{i=1}^l \bar a_{j\,i} \alpha_i + u_j \right|^p \to \min \label{Prob1}\\
&\begin{cases} 
\alpha \in \mathbb{Z}^l \setminus \{0\} \\
||\alpha||_1 \leq C.\\
\end{cases} \notag
\end{align}

where $1 \leq l \leq k$, $1 \leq C \leq M^p$, $v \in \mathbb{Z}^s$, $u \in \mathbb{Z}^m$ and $||v||_\infty \leq 2 \Delta (1 + \Delta) M^p$, $||u||_\infty \leq 2 \Delta^{1 + \log_2 3} (1+\Delta) M^p$.

Let $\sigma(l,v,u,C)$ denote the optimal value of the $Prob(l,v,u,C)$ objective function, then we trivially have
\begin{equation}\label{Sigma1}
\sigma(1,v,u,C) = \min \{ |z|^p + \sum_{i=1}^s \left| a_{i\,1} z + v_i \right|^p + \sum_{i=1}^m \left| \bar a_{i\,1} z + u_i \right|^p : z \in \mathbb{Z},\, |z| \leq C \}.
\end{equation}

The following formula gives relations between $\sigma(l,v,u,C)$ and $\sigma(l-1,v,u,C)$, correctness of the formula could be checked directly: 
\begin{equation}\label{Sigma2}
\sigma(l,v,u,C) = \min \{ f(\bar v,\bar u,z) : z \in \mathbb{Z},\, |z| \leq C,\, \bar v_i = v_i + a_{i\,l} z,\, \bar u_i = u_i + \bar a_{i\,l}z \},
\end{equation}
where
\[
f(v,u,z) = 
\begin{cases}
\sigma(l-1,v,u,C),\,\text{ for }z = 0\\
|z|^p + \min\{ \sigma(l-1,v,u,C-|z|), ||v||_p^p+||u||_p^p \},\,\text{ for }z \not= 0\\
\end{cases}.
\]

Let $\overline{Prob}(l,v,u,C)$ denote the following problem:
\begin{align}
&\sum_{i=1}^k |\alpha_i|^p + \sum_{j=1}^s \left|\sum_{i=1}^k a_{j\,i} \alpha_i + \sum_{i=1}^{\min\{j,l\}} b_{j\,i} \beta_i + v_j\right|^p + \notag\\  &+\sum_{j=1}^m \left|\sum_{i=1}^k \bar a_{j\,i} \alpha_i + \sum_{i=1}^{\min\{j,l\}} \bar b_{j\,i} \beta_i + u_j \right|^p \to \min \label{Prob2}\\
&\begin{cases}
\alpha \in \mathbb{Z}^k,\,\beta \in \mathbb{Z}^l \\
1 \leq ||\alpha||_1 + ||\beta||_1 \leq C.\\
\end{cases} \notag
\end{align}

where $1 \leq l \leq s$, $1 \leq C \leq 2 (\Delta+1) M^p$ and the values of $v,u$ are the same as in \eqref{Prob1}.

Let $\bar \sigma(l,v,u,C)$ denote the optimal value of the $\overline{Prob}(l,v,u,C)$ objective function.

Again, it is easy to see that

\begin{equation}\label{SigmaBar1}
\bar\sigma(1,v,u,C) = \min\{f(\bar v,\bar u,z) : z \in \mathbb{Z},\, |z| \leq C,\, \bar v_i = v_i + b_{i\,1} z,\, \bar u_i = u_i + \bar b_{i\,1} z \},
\end{equation}
where 

\[
f(v,u,z) = 
\begin{cases}
\sigma(k,v,u,\min\{C,M^p\}),\,\text{ for }z = 0\\
\min\{\sigma(k,v,u,\min\{C-|z|,M^p\}), ||v||_p^p + ||u||_p^p\},\,\text{ for }z \not= 0\\
\end{cases}.
\]

The following formula gives relations between $\bar\sigma(l,v,u,C)$ and $\bar\sigma(l-1,v,u,C)$:

\begin{equation}\label{SigmaBar2}
\bar\sigma(l,v,u,C) = \min\{f(\bar v, \bar u, z) : z \in \mathbb{Z},\, |z| \leq C,\, \bar v_i = v_i + b_{i\,l} z,\, \bar u_i = u_i + \bar b_{i\,l} z \},
\end{equation}
where 

\[
f(v,u,z) = 
\begin{cases}
\bar\sigma(l-1,v,u,C),\,\text{ for }z = 0\\
\min\{\bar\sigma(l-1,v,u,C-|z|), ||v||_p^p + ||u||_p^p\},\,\text{ for }z \not= 0.\\
\end{cases}
\]

\begin{theorem}
There is an algorithm to solve the problem \eqref{ISVP}, which is polynomial on $n$, size $H$ and $\Delta$. The algorithm bit-complexity is equal to \\$O(n\, d\, M^{p(2+m+\log_2 \Delta)} \Delta^{3+4m+2\log_2 \Delta} \mult(\log \Delta))$, where $\mult(k)$ is the two $k$-bit integers multiplication complexity. Since $M \leq \Delta \sqrt[p]{m+1}$ (see \eqref{MConst}), the problem \eqref{ISVP}, parameterized by a parameter $\Delta$, is included to the FPT-complexity class for fixed $m$ and $p$.
\end{theorem}
\begin{proof}
By Lemma \ref{SVBounds}, the objective function optimal value is equal to $\bar \sigma(s,0,0,2(1+\Delta)M^p)$. Using the recursive formula \eqref{SigmaBar2}, we can reduce instances of the type $\bar \sigma(s,\cdot,\cdot,\cdot)$ to instances of the type $\bar \sigma(1,\cdot,\cdot,\cdot)$. Using the formula \eqref{SigmaBar1}, we can reduce instances of the type $\bar \sigma(1,\cdot,\cdot,\cdot)$ to instances of the type $\sigma(k,\cdot,\cdot,\cdot)$. Using the formula \eqref{Sigma2}, we can reduce instances of the type $\sigma(k,\cdot,\cdot,\cdot)$ to instances of the type $\sigma(1,\cdot,\cdot,\cdot)$. Finally, an instance of the type $\sigma(1,\cdot,\cdot,\cdot)$ can be computed using the formula \eqref{Sigma1}. The bit-complexity to compute an instance $\sigma(1,v,u,C)$ is $O(C\, d\, \mult(\log \Delta) )$. The vector $v$ can be chosen using $(2\Delta(1+\Delta)M^p)^{s}$ possibilities and the vector $u$ can be chosen using $(2 \Delta^{1+\log_2 3} (1 + \Delta) M^p)^m$ possibilities, hence the complexity to compute instances of the type $\sigma(1,\cdot,\cdot,\cdot)$ is roughly 
\[
O(d\, M^{p(2+m+\log_2 \Delta)} \Delta^{1+4m+2\log_2 \Delta} \mult(\log \Delta)).
\] 
The reduction complexity of $\sigma(l,\cdot,\cdot,\cdot)$ to $\sigma(l-1,\cdot,\cdot,\cdot)$ (the same is true for $\bar \sigma$) consists of $O(C)$ minimum computations and $O(d C)$ integers multiplications of size $O(\log \Delta)$. So, the bit-computation complexity for instances of the type $\sigma(l,\cdot,\cdot,\cdot)$ for $1 \leq l \leq k$ can be roughly estimated as 
\[
O(k\, d\, M^{p(2+m+\log_2 \Delta)} \Delta^{1+4m+2\log_2 \Delta} \mult(\log \Delta)),
\]
and the bit-computation complexity for instances of the type $\bar \sigma(l,\cdot,\cdot,\cdot)$ for $1 \leq l \leq s \leq \log_2 \Delta$ can be roughly estimated as
\[
O(\log_2 \Delta\, d\, M^{p(2+m+\log_2 \Delta)} \Delta^{3+4m+2\log_2 \Delta} \mult(\log \Delta)).
\]

Finally, the algorithm complexity can be roughly estimated as
\[
O(n\, d\, M^{p(2+m+\log_2 \Delta)} \Delta^{3+4m+2\log_2 \Delta} \mult(\log \Delta)).
\]
\end{proof}

\section{The SVP for a special class of lattices}

In this section we consider the SVP \eqref{ISVP} for a special class of lattices that are induced by integral matrices without singular rank sub-matrices. Here, we inherit all notations and special symbols from the previous section.

Let the matrix $H$ have the additional property such that $H$ has no singular $n \times n$ sub-matrices. One of results of the paper \cite{AE16} states that if $n \geq f(\Delta)$, then the matrix $H$ has at most $n+1$ rows, where $f(\Delta)$ is a function that depends only on $\Delta$. The paper \cite{AE16} contains a super-polynomial estimate on the value of $f(\Delta)$. Here, we will show an existence of a polynomial estimate.

\begin{lemma}\label{NRowsHNF}
If $n > \Delta^{3 + 2\log_2 3} + \log_2 \Delta$, then $H$ have at most $n+1$ rows.
\end{lemma}
\begin{proof} Our proof of the theorem has the same structure and ideas as in the paper \cite{AE16}. We will make a small modification with usage of Lemma \ref{HNFElem}. 

Let the matrix $H$ be defined as illustrated in \eqref{HNF}. Recall that $H$ has no singular $n \times n$ sub-matrices. For the purpose of deriving a contradiction, assume that $n > \Delta^{3 + 2\log_2 3} + \log_2 \Delta$ and $H$ has precisely $n+2$ rows. Let again, as in the paper \cite{AE16}, $\bar H$ be the matrix $H$ without rows indexed by numbers $i$ and $j$, where $i,j \leq k$ and $i \not= j$. Observe, that 
\[
|\det \bar H| = |\det \underbrace{\begin{pmatrix}
a_{1\,i} & a_{1\,j}& b_{1\,1} &                   &    \\
\vdots   &\vdots    &              &   \ddots       &    \\
a_{s\,i}& a_{s\,j}&       \hdotsfor{2}          & b_{s\,s} \\
\bar a_{1\,i}& \bar a_{1\,j} &      \hdotsfor{2}          & \bar b_{1\,s} \\
\bar a_{2\,i}& \bar a_{2\,j} &      \hdotsfor{2}          & \bar b_{2\,s} \\
\end{pmatrix}}_{:={\bar H}^{ij}}|.
\]

The matrix ${\bar H}^{ij}$ is a nonsingular $(s+2)\times(s+2)$-matrix. This implies that the first two columns of ${\bar H}^{ij}$ must be different for any $i$ and $j$. By Lemma \ref{HNFElem} and the structure of HNF, there are at most $\Delta \cdot \Delta^{2(1+\log_2 3)}$ possibilities to choose the first column of ${\bar H}^{ij}$. Consequently, since $ n > \Delta^{3 + 2\log_2 3} + \log_2 \Delta$, then $k > \Delta^{3 + 2\log_2 3}$, and there must exist two indices $i \not= j$, such that $\det {\bar H}^{ij} = 0$. This is a contradiction.
\end{proof}

Using the previous theorem and Theorem \ref{SimpleSVP} of the previous section, we can develop a FPT-algorithm that solves the announced problem.

\begin{theorem}
Let $H$ be the matrix defined as illustrated in \eqref{HNF}. Let also $H$ have no singular $n \times n$ sub-matrices and $\Delta$ be the maximal absolute value of $n \times n$ minors of $H$. If $n > \Delta^{3 + 2\log_2 3} + \log_2 \Delta$, then there is an algorithm with the complexity $O(n \log n \log^2 \Delta)$ that solves the problem \eqref{ISVP}.
\end{theorem}
\begin{proof}
If $n > \Delta^{3 + 2\log_2 3} + \log_2 \Delta$, then, by the previous theorem, we have $m=1$ or $m=0$. In both cases, we have $n > \Delta^{3 + 2\log_2 3} + \log_2 \Delta > \Delta^{1+m(1+\log_2 \Delta)} + \log_2 \Delta$. The last inequality meets the conditions of Theorem \ref{SimpleSVP} and the theorem follows. 
\end{proof}  

\section{Integer linear programming problem (ILPP)}

Let $H \in \mathbb{Z}^{d \times n}$, $c \in \mathbb{Z}^n$, $b \in \mathbb{Z}^d$, $rank(H) = n$ and let $\Delta$ be the maximal absolute value of $n \times n$ minors of $H$. Suppose also that all $n \times n$ sub-matrices of $H$ are nonsingular.
 
Consider the ILPP:
\begin{equation}\label{IPP}
\max\{c^\top x : H x \leq b,\, x \in \mathbb{Z}^n \}.
\end{equation}

\begin{theorem}
Let $n > \Delta^{3 + 2\log_2 3} + \log_2 \Delta$, then the problem \eqref{IPP} can be solved by an algorithm with the complexity 
\[
O( \log \Delta \cdot n^4 \Delta^5 (n+\Delta) \cdot \mult(\log \Delta + \log n + \log ||w||_\infty) ).
\]
\end{theorem}
\begin{proof}
By Lemma \ref{NRowsHNF}, for $n > \Delta^{3 + 2\log_2 3} + \log_2 \Delta$ the matrix $H$ can have at most $n+1$ rows.

Let $v$ be an optimal solution of the linear relaxation of the problem \eqref{IPP}. Let us also suppose that $\Delta = |\det(H_{1:n\,*})| > 0$ and $H_{1:n\,*} v = b_{1:n}$. First of all, the matrix $H$ need to be transformed to the HNF. Suppose that it has the same form as in \eqref{HNF}. 

Let us split the vectors $x$ and $c$ such that $x = \dbinom{\alpha}{\beta}$ and $c = \dbinom{c_\alpha}{c_\beta}$. We note that the sub-matrix $(\bar A \, \bar B)$ from $H$ is actually a row. The problem \eqref{IPP} takes the form:
\begin{align*}
& c_\alpha^\top \alpha + c_\beta^\top \beta \to \max\\
&\begin{cases}
\alpha \leq  b_{1:k}\\
A \alpha + B \beta  \leq b_{k+1:k+s}\\
\bar A \alpha  + \bar B \beta  \leq b_{d}\\
\alpha \in \mathbb{Z}^k,\, \beta \in \mathbb{Z}^s.\\
\end{cases}\\
\end{align*}

As in \cite{AE16}, the next step consists of the integral transformation $\alpha \to b_{1:k} - \alpha$ and from an introduction of the slack variables $y \in \mathbb{Z}_+^{s}$ for rows with numbers in the range $k+1:k+s$. The problem becomes:
\begin{align*}
&- c_\alpha^\top b_{1:k} + c_\alpha^\top \alpha - c_\beta^\top \beta \to \min\\
&\begin{cases}
B \beta - A \alpha + y = \hat b\\
\bar B \beta - \bar A \alpha \leq \hat b_d\\
\alpha \in \mathbb{Z}_+^k,\, y \in \mathbb{Z}_+^s,\, \beta \in \mathbb{Z}^s,\\
\end{cases}
\end{align*}
where $\hat b = b_{k+1:k+s} - A\, b_{1:k}$ and $\hat b_d = b_d - \bar A b_{1:k}$.

We note that 
\begin{equation} \label{TardoshT}
||\binom{\alpha}{y}||_\infty \leq n \Delta
\end{equation} 
due to the classical theorem proved by Tardosh (see \cite{SCHR98,TAR86}). The Tardosh's theorem states that if $z$ is an optimal integral solution of \eqref{IPP}, then $||z - v||_\infty \leq n \Delta$. Since $A_{1:n\,*} v = b_{1:n}$ for the optimal solution of the relaxed linear problem $v$, the slack variables $\alpha$ and $y$ must be equal to zero vectors, when the $x$ variables are equal to $v$. 

Now, using the formula $\beta = B^{-1} (\hat b + A \alpha - y)$, we can eliminate the $\beta$ variables from the last constraint and from the objective function:
\begin{align*}
& -c_\alpha^\top b_{1:k} -c_\beta^\top B^{-1} \hat b + (c_\alpha^\top - c_\beta^\top B^{-1} A) \alpha + c_\beta^\top B^{-1} y \to \min\\
&\begin{cases}
B \beta - A \alpha + y = \hat b\\
(\bar B B^{*} A - \Delta \bar A) \alpha - \bar B B^{*} y \leq \Delta \hat b_d - \bar B B^{*} \hat b\\
\alpha \in \mathbb{Z}_+^k,\, y \in \mathbb{Z}_+^s,\, \beta \in \mathbb{Z}^s,
\end{cases}
\end{align*}
where the last line was additionally multiplied by $\Delta$ to become integral, and where $B^* = \Delta B^{-1}$ is the adjoint matrix for $B$.

Finally, we transform the matrix $B$ into the Smith normal form (SNF) \cite{SCHR98,STORS96,ZHEN05} such that $B = P^{-1} S Q^{-1}$, where $P^{-1}$, $Q^{-1}$ are unimodular matrices and $S$ is the SNF of $B$. After making the transformation $\beta \to Q \beta$, the initial problem becomes equivalent to the following problem:
\begin{align}
& w^\top x \to \min \label{GroupMin}\\
&\begin{cases}
G x \equiv g \,(\text{mod}\, S)\\
h x \leq h_0 \\
x \in \mathbb{Z}_+^n,\, ||x||_{\infty} \leq n \Delta,
\end{cases}\notag
\end{align}
where $w^\top = (\Delta c_\alpha^\top - c_\beta^\top B^{*} A, \; c_\beta^\top B^{*})$, $G = (P \; -PA) \mod S$, $g = P \hat b \mod S$, $h = (\bar B B^{*} A - \Delta \bar A, \; - \bar B B^{*})$, and $h_0 = \Delta \hat b_d - \bar B B^{*} \hat b$. The inequalities $||x||_{\infty} \leq n \Delta$ are additional tools to localize an optimal integral solution that follows from Tardosh's theorem argumentation (see \eqref{TardoshT}). 
  
Trivially, $||(G\;g)||_{\max} \leq \Delta$. Since $||\bar A||_{\max} \leq \Delta^{1 + \log_2 3}$, $||A||_{\max} \leq \Delta$, and $||\bar B B^*||_{\max} \leq \Delta$, we have that $||h||_{\max} \leq \Delta^2( n  + \Delta^{\log_2 3} )$. 

Actually, the problem \eqref{GroupMin} is the classical Gomory's group minimization problem \cite{GOM65} (see also \cite{HU70}) with an additional linear constraint (the constraint $||x||_{\infty} \leq n \Delta$ only helps to localize the minimum). As in \cite{GOM65}, it can be solved using the dynamic programming approach. 

To do that, let us define subproblems $Prob(l,\gamma,\eta)$:
\begin{align*}
& w_{1:l}^\top x \to \min\\
&\begin{cases}
G_{*\,1:l} x \equiv \gamma \,(\text{mod}\, S)\\
h_{1:l} x \leq \eta \\
x \in \mathbb{Z}_+^l,
\end{cases}
\end{align*}
where $l \in 1:n$, $\gamma \in \inth(G)\mod S$, $\eta \in \mathbb{Z}$, and $|\eta| \leq n^2 \Delta^3 (n+\Delta)$.

Let $\sigma(l,\gamma,\eta)$ be the objective function optimal value of the $Prob(l,\gamma,\eta)$. When the problem $Prob(l,\gamma,\eta)$ is unfeasible, we put $\sigma(l,\gamma,\eta) = +\infty$. Trivially, the optimum of \eqref{IPP} is $\sigma(n,g,\min\{h_0,\,n^2 \Delta^3 (n+\Delta)\})$.

The following formula gives the relation between $\sigma(l,*,*)$ and $\sigma(l-1,*,*)$:
\[
\sigma(l,\gamma,\eta) = \min \{\sigma(l-1,\gamma - z G_{*\,l},\eta-z h_{l}) +z w_l : |z| \leq n \Delta \}.
\]
The $\sigma(1,\gamma,\eta)$ can be computed using the following formula: 
\[
\sigma(1,\gamma,\eta) = \min \{z w_1 : z G_{*\,1} \equiv \gamma \,(\text{mod } S),\, z h_1 \leq \eta,\, |z| \leq n \Delta \}.
\]

Both, the computational complexity of $\sigma(1,\gamma,\eta)$ and the reduction complexity of $\sigma(l,\gamma,\eta)$ to $\sigma(l-1,\cdot,\cdot)$ for all $\gamma$ and $\eta$ can be roughly estimated as:
\[
O( \log \Delta \cdot n^3 \Delta^5 (n+\Delta) \cdot \mult(\log \Delta + \log n + \log ||w||_\infty) ).
\]
The final complexity result can be obtained multiplying the last formula by $n$.  

\end{proof}

\section*{Conclusion}
Here, we present FPT-algorithms for SVP instances parameterized by the lattice determinant on lattices induced by near square matrices and on lattices induced by matrices with no singular sub-matrices. In the first case, the developed algorithm is applicable for the norm $l_p$ for any finite $p \geq 1$. In the second case, the algorithm is also applicable for the $l_\infty$ norm. Additionally, we present a FPT-algorithm for ILPP instances, whose constraint matrices have no singular sub-matrices.

In the full version of the paper, we are going to extend the results related to the SVP on more general classes of norms. Next, we are going to extend result related to the ILPP for near square constraint matrices. Finally, we will present a FPT-algorithm for the simplex width computation problem.

\end{document}